\documentclass[a4paper]{article}

\pdfoutput=1

\usepackage[affil-it]{authblk}
\usepackage[fontsize=12pt]{scrextend}
\usepackage{blindtext}

\usepackage[T2A]{fontenc}
\usepackage[utf8x]{inputenc}
\usepackage[english]{babel}

\usepackage[square,sort,comma,numbers]{natbib}

\usepackage{amsmath}
\usepackage{amsthm}
\usepackage{amssymb}
\usepackage{amscd}
\usepackage[pdftex]{graphicx}
\usepackage[all]{xy}
\usepackage{euscript}

\usepackage[pdftex, colorlinks, citecolor=blue, linkcolor=black]{hyperref}

\makeatletter
\def\@maketitle{%
  \newpage
  \null
  \vskip 2em%
  \begin{center}%
  \let \footnote \thanks
    {\Large\bfseries \@title \par}%
    \vskip 1.5em%
    {\normalsize
      \lineskip .5em%
      \begin{tabular}[t]{c}%
        \@author
      \end{tabular}\par}%
    \vskip 1em%
    {\normalsize \@date}%
  \end{center}%
  \par
  \vskip 1.5em}
\makeatother

\theoremstyle{plain}
\newtheorem{theorem}{Theorem}[subsection]
\newtheorem{lemma}[theorem]{Lemma}
\newtheorem{proposition}[theorem]{Proposition}

\theoremstyle{definition}
\newtheorem{remark}[theorem]{Remark}
\newtheorem{example}[theorem]{Example}
\newtheorem{definition}[theorem]{Definition}
\newtheorem{notation}[theorem]{Notation}

\newtheorem{ex-constr}[theorem]{Construction}

\renewcommand{\dim}{\mathrm{dim}\,}

\renewcommand{\deg}{\mathrm{deg}\,}

\renewcommand{\P}{{\mathbb P}}

\newcommand{\CH}{\mathrm{CH}}

\newcommand{\Pic}{\mathrm{Pic}}
\newcommand{\R}{\EuScript R}

\renewcommand{\O}{{\mathcal O}}

\newcommand{\eff}{\mathrm{Eff}}

\begin{document}

\author{Ivan Bazhov
}
\affil{Institut de Math\'ematiques de Jussieu}
\title{On the Chow group of zero-cycles of Calabi--Yau hypersurfaces}
\date{}


\maketitle{}

\begin{abstract}
We prove the existence of a canonical zero-cycle $c_X$ on a Calabi--Yau hypersurface $X$ in a complex projective homogeneous variety. More precisely, we show that the intersection of any $n$ divisors on $X$, $n=\dim X$ is proportional to the class of a point on a rational curve in $X$.
\end{abstract}

\subsection{Introduction}
The Chow groups $\CH_i(X)$ are abelian groups generated by classes of algebraic cycles modulo rational equivalence, they are basic invariants of an algebraic variety. One has $\CH_n(X)=\mathbb Z[X]$ and $\CH_{n-1}=\Pic(X)$ for a smooth variety $X$ of dimension $n$, hence these two groups are well understood. Concerning the group of zero-cycles, it is known (cf.~\cite{Mumford, R2}) that $\CH_0(X)$ is infinite dimensional if $h^{\dim X}(X,\O_X)>0$, and the other Chow groups are even more mysterious.
Our motivating point is the following theorem, which contrasts with the results of \cite{Mumford, R2}.

\begin{theorem}[\cite{BV}]
\label{th_BV}
Let $X$ be a K3 surface.
\begin{enumerate}
\item All points of $X$ which lie on some (possibly singular) rational curve have the same class $c_X$ in $\CH_0(X)$.
\item The image of the intersection product
\begin{equation*}
\Pic(X)\otimes\Pic(X)\to \CH_0(X)
\end{equation*}
is contained in $\mathbb Zc_X$.
\item The second Chern class $c_2(X)\in\CH_0(X)$ is equal to $24c_X$.
\end{enumerate}
\end{theorem}

From now on $\CH^i(X)=\CH_{n-i}(X)\otimes \mathbb Q$. In the paper \cite{B07} Beauville put the theorem above in a more general framework and proposed a conjectural explanation for hyper-K\"ah\-ler manifolds: the class map $\mathrm{cl}:\CH^i(X)\to H^{2i}(X)$ is injective on the subring generated by divisors. A stronger version, which involves Chern classes of tangent bundles, was formulated in \cite{Voisin08}. The reader can find particular results about hyper-K\"ahler manifolds in ~\cite{Fu14, Lin15, Voisin12, Voisin15}. 

An example in \cite{B07} shows that the class map may be not injective on the ring generated by divisors  if $X$ is a Calabi--Yau variety. Fortunately, this example deals with cycles of positive dimension and still there is a hope that the class map is injective on zero-cycles or, equivalently, that the intersection of any $\dim (X)$ divisors on $X$ is proportional to a canonical zero-cycle $c_X$ in $\CH_0(X)$. The existence of such a cycle is trivial if $\Pic(X)$ is generated  by one element $H$: we can put $c_X=H^{\dim X}/\deg\left(H^{\dim X}\right)$. We aim to show the existence of $c_X$ prove the following analog of the Theorem \ref{th_BV} at least for some Calabi--Yau varieties with higher Picard rank.

\begin{theorem}
\label{main}
Let $Y$ be a complex projective homogeneous variety of dimension $n+1$ and $X$ be a general element of the anti-canonical system $|-K_Y|$ on $Y$. 
\begin{enumerate}
\item\label{m2}  There exist a constant cycle subvariety (cf.~Definition \ref{CCD}) of positive dimension. All points of $X$ which lie on such constant cycle variety have the same class $c_X$ in $\CH_0(X)$. 
\item\label{m1} The image of the intersection product 
\begin{equation*}
\Pic(X)^{\otimes n}\to\CH_0(X)
\end{equation*}
is contained in $\mathbb Zc_X$.
\item\label{m3}  The top Chern class $c_{n}(X)\in\CH_0(X)$ is proportional to $c_X$.
\end{enumerate}
\end{theorem}

The proof uses a notion of \emph{constant cycle subvarieties}, which was introduced in \cite{H14} and used in \cite{Lin15, Voisin15} to study $\CH_0(X)$ for hyper-K\"ahler manifolds. The main point of our proof is an existence of a \emph{constant cycle divisor} in an ample class (in the situation of K3 surface, the analog is the existence of rational curves in an ample class). 
Our result (\ref{m1}) also shows that the restriction to our Calabi--Yau hypersurface of any curve in a homogeneous space is proportional to the canonical cycle $c_X$ (cf.~Proposition~\ref{intersect}). 
Indeed, by the hard Lefschetz theorem, any curve class in $Y$ is an intersection of divisor classes in $Y$ and furthermore the cycle class $\CH^{n}(Y)\to H^{2n}(Y)$ is injective, so any element of $\CH^{n}(Y)$ is an intersection of divisors.  

The text is organized as follows. Firstly, we briefly recall basic facts about homogeneous spaces. Then, we construct a constant cycle divisor associated with an effective class of curves in $Y$. Using these divisors we finish the proof in the last section.

{\bf Acknowledgement.} I am grateful to my advisor Claire Voisin for her kind help and patient guidance during the work.

\subsection{Homogeneous varieties and their properties}
We recall some basic facts about homogeneous spaces (see \cite{B} and the references therein). 
For a projective homogeneous variety $Y$ the cone of effective divisors $\eff(Y)$, which is the closure of the cone of ample divisors, is a polyhedral cone. Every rational point in it represents a globally generated line bundle, which is ample if and only if the point belongs to the interior of $\eff(Y)$. Since $Y$ is Fano, its anti-canonical class $-K_{Y}$ belongs to the interior of $\eff(Y)$.

For every face $\sigma\subset \eff(Y)$ of codimension one, one defines the extremal contraction 
\begin{equation*}
\xymatrix{
Y\ar[r]^{f_\beta}&Y_\beta,
}
\end{equation*}
of the primitive class $\beta\in \sigma^\bot$ of curves, where $\beta$ is the unique integral positive generator of the rational line $\sigma^\bot$,
and $Y_\beta$ is a homogeneous variety of a smaller dimension and $f_\beta$ is a fibration. There is the natural identification
\begin{equation*}
\beta^\bot= f_\beta^*(\Pic(Y_\beta))\subset \Pic(Y).
\end{equation*}
Vice versa, for any extremal ray $\mathbb R_{>0}\beta$ there is a face $\sigma=\beta^\bot\cap\eff(Y)$ of codimension one and we can define a contraction $f_\beta$. We denote the set of all primitive effective classes of curves on extremal rays by $\R$.

A general fiber $Y_0$ of $f_\beta$ is {\it a generalized Grassmann variety} in the following sense: it is a homogeneous space $G/P$, where $P$ is a maximal parabolic subgroup of $G$. We use only one fact: $\Pic(Y_0)$ is generated by the class of an ample 
bundle $\mathcal O_Y(1)$ and the unique primitive effective class $\beta\in\CH_1(Y_0)$ has degree 1 with respect to it (we call $\beta$ the class of a line).

\subsection{Constant cycle divisors}
The notion of constant cycle subvarieties was introduced in \cite{H14} and used in \cite{Lin15} and \cite{Voisin15}. The goal of this section is to construct constant cycle divisors on  Calabi--Yau hypersurfaces.

\begin{definition}
\label{CCD}
A subvariety $Z\subset X$ of positive dimension is called \emph{a constant cycle subvariety} if for any two points $z,z'\in Z$ one has $[z]=[z']$ in $\CH_0(X)$. If, moreover, $\mathrm{codim}_X(Z)=1$, we call $Z$ a \emph{constant cycle divisor} (or CCD for short).
We will denote by $c_Z\in\CH_0(X)$ the common class of the points $z\in Z$.
\end{definition}

\begin{example}
Let $\pi:X\to \P^n$ be a ramified double cover. Clearly, the class in $\CH_0(X)$ of the degree two cycle $x_1+x_2=\pi^{-1}(p)$ does not depend on the choice of $p\in \P^n$, and if, moreover, $x_1$ and $x_2$ coincide, then the class $[x_1]=[x_2]$ is determined in $\CH_0(X)$ uniquely
. The locus, where the two preimages $x_1$ and $x_2 $ coincide is the divisor $D$ of ramification. So, $D$ is CCD on $X$.
\end{example}

The rest of the section is to expand the idea of the last example: having a family of zero-cycles on $X$, all of the same class and degree, we are looking for cycles represented by one point (with the corresponding multiplicity).

\begin{ex-constr}
\label{constr}
Let $Y$ be a Fano variety of dimension $n+1$ and assume that $Y$ has trivial $\CH_1(Y)$ group, which means the following.
\begin{itemize}
\item[($\star$)] {\it The natural class map}
\begin{equation*}
\xymatrix{
\mathrm{cl}:\CH_1(Y)\otimes\mathbb Q\ar[r]
&H^{2n}(Y,\mathbb Q)
}
\end{equation*}
{\it is injective.} 
\end{itemize}
(Varieties with trivial Chow groups include projective homogeneous spaces, smooth toric varieties, and varieties admitting a stratification by affine spaces.)

Let $\beta$ be any effective class in $\CH_1(Y)$ and $M_\beta$ be the space of all (irreducible) rational curves $C\subset Y$ representing class $\beta$. Denoting by $\mathcal C$ the universal curve we have the natural diagram:
\begin{equation}
\label{diag}
\xymatrix{
\mathcal C\ar[d]^{p}\ar[r]^{q}& Y\\
M_\beta&
}
\end{equation}
where $p$ and $q$ are the natural projections.

For a given hypersurface $X\subset Y$ we can construct the following variety:
\begin{equation*}
V_{X,\beta}=
\{(C,x)\in\mathcal C: C\cap X=\deg(C\cap X)\cdot x\},
\end{equation*}
where the quality is the equality of zero-cycles on $C$ (or even of subschemas of $C$ if $C$ is smooth). The subvariety $p(V_{X,\beta})$ in $M_\beta$ describes all curves intersecting $X$ in one point with the maximal multiplicity. Due to our assumption ($\star$), the class of the zero-cycle $C|_X$ in $\CH_0(X)$ does not depend on the point $C$ in $M_\beta$. By \cite{R}, a family of torsion cycles in a variety has to be constant, hence the subvariety $q(V_{X,\beta})\subset X$ is a constant cycle subvariety. 

If $X\in|-K_Y|$ is general, so that $X$ is a smooth Calabi--Yau hypersurface in $Y$, and furthermore there is no obstruction for deformation of curves in $Y$, then denoting $d=\deg{C\cap X}$ we have
\begin{multline*}
\dim V_{X,\beta}\geq\dim M_\beta-(d-1)=(-K_{Y}\cdot \beta+(n+1)-3)-(d-1)\\=(X\cdot C+\dim X-2)-d+1=\dim X-1,
\end{multline*}
and we expect that $\overline{q(V_{X,\beta})}$ is a CCD on $X$. 
\end{ex-constr}

\begin{notation}
Continuing with settings of Construction \ref{constr}, we put
\begin{equation*}
H_\beta=\overline{q(V_{X,\beta})}.
\end{equation*}
\end{notation}

\begin{lemma}
\label{l_eff_1}
Let $Y$ be a projective homogeneous variety with Picard number one and dimension at least two and let $\beta$ be the class of a line.  If $X\in|-K_Y|$ is general, then $H_{\beta}$ is a non-empty effective divisor on $X$.
\end{lemma}

\begin{proof} 
First of all, $H_\beta$ can not be the whole of $X$, because it would imply that all points in $X$ are rationally equivalent. In order to show that $H_\beta$ is a non-empty effective divisor, it thus suffices to show that its virtual class is non-zero.  

We denote the class of the standard polarization on $Y$ by $H_1$ and notation as in (\ref{diag}), let $\tilde H_1=q^*H_1$ be the corresponding divisor on the universal curve. We also put $H_2:=p_*(\tilde H_1^2)=p_*q^*H_1^2$, let $\tilde H_2=p^*H_2$. Denote  by $K_{rel}$ the class of the relative cotangent bundle on the universal curve, it is easy to see that $K_{rel}=-2\tilde H_1+\tilde H'_2$, where $\tilde H'_2=p^*H'_2$ for some divisor $H'_2$ on $M_\beta$.

Let us show that $H_2=H'_2$. By Grothendieck--Riemann--Roch formula,
\begin{equation*}
c_1\left(R^0p_*\mathcal O_{\mathcal C}(\tilde H_1)\right)=p_*\left(\tilde H_1^2+\frac{5}{6}\left(\tilde H_1^2-\tilde H_1\tilde H'_2\right)+\frac{1}{12}\tilde H^2_2\right),
\end{equation*}
but  the left-hand side can be calculated geometrically: it is the locus of all curves $C$ in $M_\beta$ such that  the restrictions of two sections in $H^0(Y, \mathcal O_{Y}(H_1))$  to $C$ coincide. Hence the left-hand side is equal to $p_*q^*(H_1^2)=p_*\left(\tilde H_1^2\right)$. Since $p_*\left(\tilde H_2^2\right)=0$, we get
\begin{equation*}
p_*\left(\frac 56\left(\tilde H_1^2-\tilde H_1\tilde H'_2\right)\right)=0.
\end{equation*}
Since $\tilde H_1$ has degree one on fibers of $p:\mathcal C\to M_\beta$, we get $p_*\left(\tilde H_1\tilde H'_2\right)=H'_2$ and
\begin{equation*}
H'_2=p_*\left(\tilde H_1\tilde H'_2\right)=p_*\left(\tilde H_1^2\right)=H_2.
\end{equation*}

The subvariety $V_{X,\beta}\subset\mathcal C$ is the intersection of $d$ divisors of classes ($0\leq r\leq d-1$)
\begin{equation*}
d\tilde H_1+rK_{rel},
\end{equation*}
where the divisor number $r$ is to say that the $r-$th derivative of the restriction $f_{X}|_C$ is zero at a point on the curve $C$, where $f_X$ is the defining equation for $X$; all together they mean that $f_X|_C$ is one point with the maximal multiplicity. So, the degree of $H_\beta=q_*(V_{X,\beta})$ can be calculated as the degree of the following intersection on $Y$ (we recall $\dim X=n$, $\dim Y=n+1$):
\begin{multline*}
H_1^{n-1}\cdot q_*\left[d\tilde H_1\cdot \left(d\tilde H_1+K_{rel}\right)\cdot \ldots\cdot \left(d\tilde H_1+(d-1)K_{rel}\right)\right]\\=
q_*\left[\tilde H_1^{n-1}\cdot d\tilde H_1\cdot \left(d\tilde H_1+K_{rel}\right)\cdot \ldots\cdot \left(d\tilde H_1+(d-1)K_{rel}\right)\right]\\=
q_*\left[d\tilde H_1^{n} \prod\limits_{r=1}^{d-1}\left(d\tilde H_1+r\left(-2\tilde H_1+\tilde H_2\right)\right)\right]\\=
q_*\left[d! \left(\tilde H_1^{n}\tilde H_2^{d-1}+ \sum\limits_{r=1}^{d-1}\left(\frac{d-2r}{r}\right)\tilde H_1^{n+1}\tilde H_2^{d-2} \right)\right].
\end{multline*}

Clearly, points of $p_*(\tilde H_1^{n+1})=p_*q^*(H_1^{n+1})$ in $M_\beta$ correspond to curves which pass through a point of the intersection $H_1^{n+1}$ in $Y$. Since $H_2$ is ample on $M_\beta$, the intersection $\tilde H_1^{n+1}\tilde H_2^{d-2}$ is not empty. To finish the proof, we note that the sum $\sum_{r=1}^{d-1}(d/r-2)$ is strictly positive if $d>2$.
\end{proof}

\begin{lemma}
\label{t_eff}
Let $Y$ be a homogeneous variety and $X\in|-K_Y|$ be general.
If $f_\beta:Y\to Y_\beta$ is not a $\P^1-$fibration for some $\beta\in\R$, then $H_\beta$ is a non-empty effective CCD on $X$.
\end{lemma}

\begin{proof}
It is enough to show that $V_{X,\beta}$ is non-empty of an expected dimension and a general fiber of $q:V_{X,\beta}\to q(V_{X,\beta})$ is finite.

A general fiber $Y_0$ of $f_\beta: Y\to Y_\beta$ is a Schubert variety in $Y$, hence the restriction map $H^0(Y, \mathcal O_Y(X))\to H^0(Y_0, \mathcal O_Y(X)|_{Y_0})$ is surjective (\cite[Section 3]{BK} or \cite[Theorem 2.3.1]{B}). Applying the previous lemma to $Y_0$ and $X|_{Y_0}$ we see that $V_{X,\beta}$ is non-empty and projection $q$ has some finite fibers, hence a general fiber is finite. The dimension of $V_{X,\beta}$ is expected because the dimension of $M_\beta$ is expected.
\end{proof}

\begin{example}
\label{ex2}
Let $Y=\P^2\times S$, where $S$ is a Fano variety of positive dimension and let $\beta$ be equal to $[l\times\mathrm{point}]$, where $l$ is a line in $\P^2$. If $X$ is a general Calabi--Yau hypersurface in $Y$ then the restriction of $H_\beta$ to a fiber of $\P^2\times S\to S$ is nine Weierstrass points on a plane elliptic curve. In particular, $H_\beta$ is non-empty and effective.
\end{example}

\begin{example}
\label{ex}
Let $Y$ be as in Construction \ref{constr} and assume that
\begin{equation*}
\xymatrix{
f_\beta:Y\ar[r]&Y_\beta
}
\end{equation*}
is a $\P^1-$fibration (again $\beta$ is the class of contracted curves, i.e, the class of a fiber). If $X$ is a general Calabi--Yau hypersurface then $H_\beta$ has class $(-K_Y+K_{Y/Y_\beta})|_X=f_\beta^*(-K_{Y_\beta})|_X$ in $\Pic(X)$ and $H_\beta$ is trivial on fibers of $f_\beta$ (in contrast to Lemma \ref{t_eff} and Example \ref{ex2}).
\end{example}

\begin{example}[The curve of hyperflexes, cf.~\cite{H14} and \cite{W}]
In this example $Y=\P^3$, $X$ is a generic quartic, and $\beta$ is the class of a line. The curve $C_{hf}=q(V_{X,\beta})$ is a singular and irreducible curve in the linear system $|\mathcal O_X(20)|$ of geometric genus 201. In particular, $C_{hf}$ is not rational.  

Geometrically, $C_{hf}$ arises as follows. Recall that for a quartic $X\subset \P^3$ a line $l\subset \P^3$ is called bitangent of $X$ if at every point $x\in X\cap l$ the intersection multiplicity is at least two, a bitangent $l$ is a hyperflex if there is a unique point of the intersection. We can consider the universal family of bitangents:
\begin{equation*}
\xymatrix{
{\phantom{\beta}\mathcal C\supset B_X}\ar[r]^{q_0}
\ar@<2,5ex>[d]^{p_0}&X\subset \P^3\\
M_\beta\supset F_X
}
\end{equation*}
here $F_X\subset M_\beta$ is the subvariety of all bitangents and $B_X$ is the variety of pairs $(l,x)$, where $x$ is a point of contact of $l$ and $X$. The curve $C_{hf}$ can be defined as $q_0(p_0^{-1}(D_{hf}))$, where $D_{hf}$ is the ramification  divisor of the degree two map $p_0: B_X\to F_X$.
\end{example}

\subsection{The proof of Theorem \ref{main}}
Let us first prove the following.

\begin{proposition}
\label{intersect}
In the setting of Theorem \ref{main}, there exists a class $c_X$ in $\CH_0(X)$ such that the restriction $C|_X$ of any curve in $Y$ is proportional to $c_X$ in $\CH_0(X)$.
\end{proposition}

\begin{proof}
Because of the linearity of the intersection, it is enough to prove the statement only for curves with class $\beta\in\R$. By construction of $H_\beta$, the corresponding zero-cycles can be represented by points (with multiplicities) on $H_\beta$ for $\beta\in\R$. Let us show that the class $c_{H_\beta}$ does not depend on the choose of $\beta\in\R$. 

{\it Case 1: $\dim Y_\beta<n$ for any $\beta\in \R$.}
We may assume that dimension of $X$ is at least three and thus by Grothendieck--Leftschetz theorem, we can identify $\Pic(X)$ and $\Pic(Y)$ by the restriction map. Let $H_\beta$ be the restriction of a divisor $\tilde H_\beta$ on $Y$. We claim that the divisor 
\begin{equation*}
\tilde H:=\sum\limits_{\beta\in \R} \tilde H_\beta
\end{equation*}
is ample on $Y$. This is equivalent to $\tilde H\cdot \beta >0$ for all $\beta\in \R$. To see this, it suffices to prove the inequalities
\begin{equation*}
\tilde H_\beta|_{\beta'}\geq 0\quad \mbox{and}\quad
\tilde H_\beta|_\beta>0
\end{equation*}
for any $\beta,\beta'\in\R$. Since the fibers of $f_{\beta'}:X\to Y_{\beta'}$ have positive dimension, 
any curve $C$ contained in the fibers of $X\to Y_{\beta'}$ has its class in $Y$ a non-zero multiple $N\beta'$. The curve $C$ can be chosen to be movable in $X$ and thus we get
\begin{equation*}
N\tilde H_{\beta}\cdot \beta'=H_\beta\cdot C\geq 0,
\end{equation*}
because $H_\beta$ is effective in $X$.
The second inequality follows from Lemma \ref{l_eff_1}. Thus the claim is proved.

The divisor $H:=\tilde H|_X$ is thus ample on $X$, hence it is connected and any point on divisors $H_\beta$ represents the same class $c_{X}=c_{H_{\beta}}$ in $\CH_0(X)$.

{\it Case 2: $\dim Y_\beta=n$ for some $\beta\in\R$.} In this case, $f_\beta|_X:X \to Y_\beta$ is a double cover (out of some codimension two subvariety, where $X$ contains a whole fiber of $f_\beta$). Due to Example \ref{ex}, the divisor $H:=H_\beta$ is a pull-back of an ample divisor on $Y_\beta$, so $H$ meets $H_{\beta'}$ for any $\beta'$. It thus follows that $c_{H_{\beta'}}=c_{H_{\beta}}$ for any $\beta'$ and we can take $c_X$ as the class of a point on $H$.
\end{proof}

Note that in a situation where $Y$ is not homogeneous, but still satisfies assumption ($\star$) the following alternative lemma could be used.

\begin{lemma}
\label{l2}
If there exist an ample divisor $\tilde H$ on $Y$ such that $\tilde H\cap X$ is a CCD on $X$ with associated class $c_H$, 
then for any curve $C$ on $Y$, $C|_X$ is proportional to $c_H$ and, in particular, the intersection of any $n$ divisors is proportional to $c_H$. 
\end{lemma} 

\begin{proof}
We have the following commutative diagram:
\begin{equation*}
\xymatrix{
{\tilde H^{n-1}:\CH^1(Y)}\ar[r]
\ar@<2,5ex>[d]& \CH^n(Y)\ar[d]\\
{\phantom{CCCC}} H^{2}(Y)\ar[r]& H^{2n}(Y)
}
\end{equation*}
The vertical arrows are isomorphisms due to condition ($\star$), the bottom arrow is also an isomorphism by the hard Lefschetz theorem, hence the top arrow is an isomorphism. Therefore the restriction $C|_X$ is equal to $(L\cdot \tilde H^{n-1})|_X$ for some divisor $L$ on $Y$. But the last restriction can be represented by a zero-cycle with the support on $H$, hence is proportional to $c_H$.
\end{proof}

\begin{proof}[Proof of Theorem \ref{main}]

The divisor $H$ constructed in the proof of Lemma \ref{intersect} does not intersect a constant cycle subvariety $Z$ only if $Z$ is a rational curve contracted by $f_\beta$ and $\dim Y_\beta=n$. But in this case the class of a point on $Z$ is proportional to $\beta|_X$ and hence is proportional to $c_X$.
Let $\tilde D_i$ be a divisor on $Y$ such that $\tilde D_i|_X=D_i$ in $\Pic(X)$; we then have in $\CH_0(X)$
\begin{equation*}
D_1\cdot \ldots\cdot D_n=\left.\left(\tilde D_1\cdot\ldots\cdot \tilde D_n\right)\right|_X.
\end{equation*}
The intersection in the right-hand side is a (reducible) curve in $Y$, we proved that the restriction of any curve in $Y$ to $X$ is proportional to $c_X$ in Proposition \ref{intersect}.

The top Chern class of $X$ can be represented by a combination of Chern classes of $Y$ and Chern classes of line bundle $\mathcal O_Y(X)$. This is again a restriction of some (reducible) curve on $Y$.
\end{proof}

\begin{remark}
There is another strategy to prove Theorem \ref{main} \ref{m1}. If $\dim Y_\beta<n$ for some $\beta$ then 
\begin{equation*}
D_1\cdot\ldots\cdot D_n=0
\end{equation*}
for any divisors $D_i$ on $Y_\beta$. The same is true for divisors $f_\beta^*(D_i)|_X$ on $X$. Since $H_\beta$ and $f_\beta^*(\Pic(Y_\beta))|_X$ generate $\Pic(X)\otimes\mathbb Q$, any intersection of $n$ divisors on $X$ is proportional to a zero-cycle with support on $H_\beta$ and hence is proportional to $c_X$.
\end{remark}

\begin{remark}
The method could be presumably generalized to the case where $Y$ is a Fano variety satisfying assumption ($\star$) using Lemma \ref{l2}.
\end{remark}

\bibliographystyle{abbrvnat}
\bibliography{bazhov_CY_biblio}

\medskip

\medskip

Institut de Math\'ematiques de Jussieu, 4 Place Jussieu, 75005 France, \texttt{ibazhov@gmail.com}

\end{document}